\documentclass[12pt,reqno]{amsart}
\usepackage{amssymb}
\usepackage{mathrsfs,mathtools}
\usepackage{palatino}
\usepackage{bbm}
\usepackage{pictex}
\usepackage{amscd}
\usepackage{latexsym}
\usepackage{amssymb}
\usepackage{enumitem}
\usepackage{eucal}
\usepackage{MnSymbol}
\usepackage{eufrak}
\usepackage{verbatim}
\usepackage{bbm}
\usepackage{hyperref}
\hypersetup{colorlinks,linkcolor=blue,urlcolor=cyan,citecolor=red}
\usepackage{pictex}
\usepackage{amscd}
\usepackage{latexsym}
\usepackage{amssymb}
\usepackage{eucal}
\usepackage{eufrak}
\usepackage{verbatim}

\newcommand{\p}{\partial}

\newcommand{\FF}{\mathbb{F}}

\newcommand{\NN}{\mathbb{N}}

\newcommand{\ZZ}{\mathbb{Z}}

\newcommand{\cO}{\mathcal{O}}

\newcommand{\cZ}{\mathcal{Z}}

\newcommand{\fg}{\mathfrak{g}}

\newcommand{\kk}{\mathbbm{k}}

\newcommand{\fsl}{\mathfrak{sl}}

\newcommand{\Wl}{W_{\ell}}

\DeclareMathOperator{\Char}{char}

\DeclareMathOperator{\End}{End}

\DeclareMathOperator{\Ker}{Ker}

\numberwithin{equation}{section}
\newtheorem{Theorem}{Theorem}[section]
\newtheorem{Lemma}[Theorem]{Lemma}

\newtheorem{Proposition}[Theorem]{Proposition}

\theoremstyle{Theorem}

\newtheorem*{thm*}{Theorem}
\newtheorem*{thm**}{Corollary}
\newtheorem*{thm***}{Theorem B}
\theoremstyle{remark}

\numberwithin{equation}{section}

\begin{document}
\title[Truncated current Witt algebra]{Simple modules over truncated current Lie algebras of the Witt algebra}
\author[H. Chang, R. Hou \lowercase{and} J. Hu]{Hao Chang, Ruiying Hou\lowercase{and} Jinxin Hu}
\address[H. Chang]{School of Mathematics and Statistics,
Central China Normal University, Wuhan 430079, China}
\email{chang@ccnu.edu.cn}
\address[R. Hou]{School of Mathematics and Statistics,
Central China Normal University, Wuhan 430079, China}
\email{houry1998@163.com}
\address[J. Hu]{Department of Mathematics,
Soochow University, Suzhou 215506, China}
\email{20244007004@stu.suda.edu.cn}
\date{\today}
\subjclass[2020]{Primary 17B50}
%%%%%%%%%%%%%%%%%%%%%%%%%%%%%%%%%%%%%%%%%%%%%%%%%%%%%%%%%%%%%%
\begin{abstract}
Let $\kk$ be an algebraically closed field of characteristic $p>3$,
and let $W$ denote the $p$-dimensional Witt algebra--the
first example of a non-classical simple Lie algebra.
For a non‑negative integer $\ell$,
consider the associated truncated current Lie algebra
$W_\ell=W\otimes\kk[t]/(t^{\ell+1})$.
In this paper,
we first study simple $W_\ell$-modules having $p$-character $\chi$ of height at most one,
and provide a complete classification of such modules up to isomorphism. 
We then investigate a family of simple $W_\ell$-modules
whose $p$-characters have height greater than one.
\end{abstract}
\maketitle
%%%%%%%%%%%%%%%%%%%%%%%%%%%%%%%%%%%%%%%%%%%%%%%%%%%%%%%%%%%%%%
%%%%%%%%%%%%%%%%%%%%%%%%%%%%%%%%%%%%%%%%%%%%%%%%%%%%%%%%%%%%%%
%%%%%%%%%%%%%%%%%%%%%%%%%%%%%%%%%%%%%%%%%%%%%%%%%%%%%%%%%%%%%%
\section{Introduction}
Let $\fg$ be a complex semisimple Lie algebra and $\ell$ be a nonnegative integer.
The {\it truncated current algebra} is the Lie algebra 
\[
    \fg_\ell:=\fg\otimes \mathbb{C}[t]/(t^{\ell+1}).
\]
Truncated current Lie algebras have appeared in numerous parts of the literature in
recent years, and
the representation theory of such algebras has been intensively studied recently.
The first truncated current Lie algebra $\fg_1$ appeared in the work of Takiff \cite{Takiff71}, and so they are often referred to as {\it Takiff Lie algebras}.
In \cite{MS19},
Mazorchuck and Sörderberg introduced a version of category $\mathcal{O} $ for Takiff $\fsl_2$.
Later, Chaffe \cite{Ch23} made a thorough study of category $\cO$ for all Takiff Lie algebras.
These results were generalized by Chaffe and Topley \cite{CT24}.
They studied the category $\cO$ for the more general current Lie algebras $\fg_\ell$.
Let $\kk$ be an algebraically closed field of positive characteristic.
Recently, Chaffe and Topley \cite{CT23} developed the representation theory of 
truncated current Lie algebras $\fg_\ell$ over $\kk$ in the case where $\fg$ is the Lie algebra of a standard reductive group.

In this paper, we initiate a study of modular representations of the {\it truncated current Witt algebra} $\Wl$ over $\kk$.
By definition, it is the Lie algebra over the truncated current algebra:
\[
    \Wl:=W\otimes \kk[t]/(t^{\ell+1}),
\]
where $W$ is the $p$-dimensional {\it Witt algebra}.
As is well known,
any finite dimensional simple Lie algebra over an algebraically closed field of characteristic $p>5$ is either classical or of Cartan type (cf. \cite{Strade04}).
Among the later, the Witt algebra $W$ is the simplest.
It was found by Witt as the first example of non-classical simple Lie algebra in the 1930’s.
The irreducible $W$-modules were determined by Chang in \cite{Zhang41}.
Strade \cite{Strade77} gave proofs many of Chang's results in a different approach.
Feldvoss and Nakano \cite{FN98} exploited further the representation and cohomology theory for the Witt algebra $W$.
Shu \cite{Shu98} determined the irreducible modules for the Zassenhaus algebra.
As $\Wl$ is restricted,
every simple $\Wl$-module admits a {$p$-character} $\chi\in\Wl^*$ (see \cite[Chapter 5, Theorem 2.5]{SF88}).
We only have to consider the simple modules for the {\it reduced enveloping algebra} $U_\chi(\Wl)$.
Following \cite{Strade77}, 
we classify the simple $\Wl$-modules according to the {\it height} of their $p$-characters.
Precisely
speaking, 
we first give a complete classification of the isomorphism classes of simple $\Wl$-modules with $p$-characters of height at most one.
By employing the main theorem of \cite{Strade77}, 
we also study a class of simple $\Wl$-modules with $p$-characters of height greater than one.
We hope our results may offer some insight
into the general picture of the representation theory for  
truncated current algebras associated to other restricted Lie algebras of Cartan type.

We organize this article in the following manner.
In Section \ref{section name -1},
we recall some results about the Witt algebra and introduce the $\chi$-reduced Verma modules for $\Wl$.
It is shown that each simple $\Wl$-module is the homomorphism image of some $\chi$-reduced Verma module under the assumption that the height $r(\chi)\leq 1$.
Section \ref{section name-2} is devoted to the study of simple $\Wl$-modules having $p$-character $\chi$ of height at most one.
We give a complete classification of the isomorphism classes of simple $\Wl$-modules (Theorem \ref{Theorem: height 0 iso classes}, Theorem \ref{thm:height 1 iso 1}, Theorem \ref{thm:height 1 iso 2}).
In Section \ref{section name-3}, we further study a class of simple $U_\chi(\Wl)$-modules with $r(\chi)>1$.
In this case, $U_\chi(\Wl)$ has only one simple module (up to isomorphism).
We constructed the simple $U_\chi(\Wl)$-module and determine its dimension (Theorem \ref{thm:height greater than one}).

\bigskip
\emph{Throughout this paper, $\kk$ denotes an algebraically closed field of characteristic $\Char(\kk)=:p>3$.}
%%%%%%%%%%%%%%%%%%%%%%%%%%%%%%%%%%%%%%%%%%%%%%%%%%%%%%%%%%%%%%
%%%%%%%%%%%%%%%%%%%%%%%%%%%%%%%%%%%%%%%%%%%%%%%%%%%%%%%%%%%%%%
%%%%%%%%%%%%%%%%%%%%%%%%%%%%%%%%%%%%%%%%%%%%%%%%%%%%%%%%%%%%%%
\section{Preliminaries}\label{section name -1}
%%%%%%%%%%%%%%%%%%%%%%%%%%%%%%%%%%%%%%%%%%%%%%%%%%%%%%%%%%%%%%
%%%%%%%%%%%%%%%%%%%%%%%%%%%%%%%%%%%%%%%%%%%%%%%%%%%%%%%%%%%%%%
%%%%%%%%%%%%%%%%%%%%%%%%%%%%%%%%%%%%%%%%%%%%%%%%%%%%%%%%%%%%%%
\subsection{Witt algebra}
Let $A(1):=\kk[X]/(X^p)$ be the truncated polynomial algebra in one indeterminate and let $x$ denote the class of $X$ in $A(1)$.
The {\it Witt algebra} $W$ is defined as the Lie algebra of derivations of $A(1)$.
We have a canonical basis $\{e_i;~-1\leq i\leq p-2\}$ for $W$,
where $e_i=x^{i+1}\p$ and $\p$ is the usual differential operator satisfying $\p(x^i)=ix^{i-1}$.
The Lie bracket in $W$ is given by 
\[
[e_i,e_j]=(j-i)e_{i+j},~\quad -1\leq i,j\leq p-2,
\]
with $e_{i+j}:=0$ if $i+j\notin\{-1,\dots, p-2\}$ by definition.
Thus, we have a gradation
\begin{align}\label{Z-grading of W}
W=\bigoplus\limits_{i=-1}^{p-2}W_{[i]},   
\end{align}
where $W(1)_{[i]}=\kk e_i$.
We set $W_{(i)}:=\oplus_{j\geq i}W_{[j]}$.
Note also that $W$ is a {\it restricted} Lie algebra,
the $p$-map being given by raising to the $p$th power using the associative product in $W\subseteq\End(A(1))$.
In particular, $e_i^{[p]}=\delta_{i,0}e_i$.
%%%%%%%%%%%%%%%%%%%%%%%%%%%%%%%%%%%%%%%%%%%%%%%%%%%%%%%%%%%%%%
%%%%%%%%%%%%%%%%%%%%%%%%%%%%%%%%%%%%%%%%%%%%%%%%%%%%%%%%%%%%%%
%%%%%%%%%%%%%%%%%%%%%%%%%%%%%%%%%%%%%%%%%%%%%%%%%%%%%%%%%%%%%%
\subsection{Simple $W$-modules}\label{subsection: simple W1modules}
For $\chi\in W^*$,
let $U_\chi(W)$ be the corresponding {\it reduced enveloping algebra} (cf. \cite[Section 2.7]{Jantzen98}).
Strade \cite{Strade77} defines
the {\it height} $r(\chi)$ of $\chi$ by
\[
r(\chi):=\min\{i;~-1\leq i\leq p-2~{\rm and}~\chi(W_{(i)})=(0)\}. 
\]
The simple $W$-modules were determined by Zhang in \cite{Zhang41}.
We describe here very briefly the simple $W$-modules with a $p$-character $\chi$ of height $-1,0,1$.
In all these cases, $\chi$ vanishes on the $p$-unipotent radical $W_{(1)}$ of $W_{(0)}$.
In particular, $W_{(1)}$ acts trivially on every simple $U_\chi(W_{(0)})$-module.
Hence the isomorphism classes of simple $U_\chi(W_{(0)})$-modules are in one-to-one correspondence with the {\it weights}
\begin{align}\label{Lambda chi}
\Lambda(\chi):=\{\lambda\in\kk;~\lambda^p-\lambda=\chi(e_0)^p\}.
\end{align}
If $r(\chi)\leq 0$, then $\chi(W_{(0)})=(0)$.
It follows that $\Lambda(\chi)$ coincides with the prime field $\FF_p$ of $\kk$.
For each $\lambda\in\Lambda(\chi)$ consider the
{\it $\chi$-reduced Verma module} 
\[
Z_\chi(\lambda):=U_\chi(W)\otimes_{U_\chi(W_{(0)})}\kk_\lambda 
\]
where $\kk_\lambda:=\kk.v_\lambda$ denotes the one-dimensional $U_\chi(W_{(0)})$-module via
\[
e_0.v_\lambda=\lambda v_\lambda,~\quad e_i.v_\lambda=0~\forall 1\leq i\leq p-2.
\]
The following Theorem is due to Zhang (\cite[Hauptsatz 2']{Zhang41}, see also \cite[Theorem A]{FN98}).
\begin{Theorem}\label{Thm:Zhangherui results}
Let $\chi\in W^*$ with $r(\chi)\leq 1$.
\begin{enumerate}
\item[(1)] If $r(\chi)=-1$, then there are $p$ isomorphism classes of simple $U_\chi(W)$-modules.
These modules are represented by the one-dimensional trivial $W$-module $\kk$, the $(p-1)$-dimensional module $Z_\chi(p-1)/{\rm Soc}_{U_\chi(W)}Z_\chi(p-1)$
and the $p$-dimensional modules $Z_\chi(\lambda)$ for $\lambda\in\{1,2,\dots,p-2\}$.
\item[(2)] If $r(\chi)=0$, then there are $p-1$ isomorphism classes of simple $U_\chi(W)$-modules each of dimension $p$ and represented by $Z_\chi(\lambda)$ for $\lambda\in\{0,1,\dots,p-2\}$.
\item[(3)] If $r(\chi)=1$, then then there are $p$ isomorphism classes of simple $U_\chi(W)$-modules
each of dimension $p$ and represented by $Z_\chi(\lambda)$ for $\lambda\in\Lambda(\chi)$.
\end{enumerate}
\end{Theorem}

%%%%%%%%%%%%%%%%%%%%%%%%%%%%%%%%%%%%%%%%%%%%%%%%%%%%%%%%%%%%%%
\subsection{Truncated current Lie algebras}
We fix a non-negative integer $\ell$.
The {\it truncated current Lie algebra} associated to $W$ is defined via
\[
 W_\ell:=W\otimes\kk[t]/(t^{\ell+1}).    
\]
We make the notation $x\otimes t^i=xt^i$.
Then the elements $e_it^j$ with $i=-1,0,\dots,p-2$ and $0\leq j\leq \ell$ form a basis of $W_\ell$.
In particular, $\dim\Wl=(\ell+1)p$.
The bracket relation and $p$-mapping are given by 
\[
[e_it^k,e_jt^l]=(j-i)e_{i+j}t^{k+l},\quad (e_it^k)^{[p]}=\delta_{i,0}e_it^{kp}.
\]
It follows that $\ZZ$-grading of $W$ \eqref{Z-grading of W} induces a $\ZZ$-grading of $\Wl$:
\[
W_\ell=\bigoplus\limits_{i=-1}^{p-2}W_{\ell,[i]},      
\]
where $W_{\ell,[i]}=\kk e_i\otimes\kk[t]/(t^{\ell+1})$.
We consider the associated descending filtration 
\[
\Wl=W_{\ell,(-1)}\supseteq W_{\ell,(0)}\supseteq\cdots\supseteq W_{\ell,(p-2)}\supseteq (0)
\]
defined via $W_{\ell,(i)}:=\sum_{j\geq i}W_{\ell,[j]}$.   
%%%%%%%%%%%%%%%%%%%%%%%%%%%%%%%%%%%%%%%%%%%%%%%%%%%%%%%%%%%%%%
\subsection{$\chi$-reduced Verma modules}
Similar to Section \ref{subsection: simple W1modules},
the {\it height} $r(\chi)$ of a $p$-character $\chi\in\Wl^*$ is defined as the smallest integer $i$ with $\chi(W_{\ell,(i)})=(0)$.
Observe that $r(\chi)=-1$ if and only if $\chi=0$.

Now suppose that $\chi\in\Wl^*$ satisfies $\chi(W_{\ell,(1)})=(0)$,
i.e., that the height of $\chi$ is at most equal to $1$.
We can make any $U_\chi(W_{\ell,[0]})$-module $M$ into a $U_\chi(W_{\ell,(0)})$-module letting any $x\in W_{\ell,(1)}$ act as $0$ on $M$.
We can then construct the induced $U_\chi(\Wl)$-module
\[
\cZ_\chi(M):=U_{\chi}(\Wl)\otimes_{U_\chi(W_{\ell,(0)})}M.
\]
Since $W_{\ell,[0]}$ is abelian,
it follows that the irreducible $U_\chi(W_{\ell,[0]})$-module is one-dimensional.
Observe that $W_{\ell,[0]}=\kk e_0\oplus \sum_{i=1}^n\kk e_0t^i$ and $\sum_{i=1}^n\kk e_0t^i$ is a unipotent subalgebra of $W_{\ell,[0]}$.
Therefore the isomorphism classes of one-dimensional $U_\chi(W_{\ell,[0]})$-modules still are parameterized by the set $\Lambda(\chi)$ \eqref{Lambda chi},
that is,
there is for any $\lambda\in\Lambda(\chi)$ a simple $U_\chi(W_{\ell,[0]})$-module $\kk_\lambda$ generated by an element $v_\lambda$ such that $e_0.v_\lambda=\lambda v_\lambda$,
and the action of $e_0t^i.v_\lambda$ with $1\leq i\leq \ell$ is completely determined by the $p$-character $\chi$ (see \cite[Section 3.3]{Jantzen98}).
We will let $e_0t^i.v_\lambda=\chi_iv_\lambda$.
In particular, we have $\chi_\ell=\chi(e_0t^{\ell})$.

On the other hand,
since $W_{\ell,(1)}$ is a unipotent ideal of $W_{\ell,(0)}$
and the Lie algebra $W_{\ell,(0)}/W_{\ell,(1)}\cong W_{\ell,[0]}$ is abelian,
we see that the simple $U_\chi(W_{\ell,(0)})$-modules are the $\kk_\lambda$ with $\lambda\in\Lambda(\chi)$.
We shall use the simplified notation
\begin{align}\label{z chi lambda}
\cZ_\chi(\lambda):=\cZ_\chi(\kk_\lambda),
\end{align}
and call it the {\it $\chi$-reduced Verma module} for $\Wl$.

Let $M$ be a simple $U_\chi(\Wl)$-module.
Then $M$ contains a simple $U_\chi(W_{\ell,(0)})$-module,
say $\kk_\lambda$.
Hence, by Frobenius reciprocity, we have a non-zero homomorphism
\[
  \cZ_\chi(\lambda)=U_{\chi}(\Wl)\otimes_{U_\chi(W_{\ell,(0)})}\kk_\lambda\twoheadrightarrow M.
\]
This shows:
\begin{Proposition}\label{prop: simpe module is image of zchilambda}
 Let $\chi\in\Wl^*$ with $r(\chi)\leq 1$.
 Then each simple $U_\chi(\Wl)$-module is the homomorphic image of some $\cZ_\chi(\lambda)$ with $\lambda\in\Lambda(\chi)$.
\end{Proposition}
%%%%%%%%%%%%%%%%%%%%%%%%%%%%%%%%%%%%%%%%%%%%%%%%%%%%%%%%%%%%%%
%%%%%%%%%%%%%%%%%%%%%%%%%%%%%%%%%%%%%%%%%%%%%%%%%%%%%%%%%%%%%%
%%%%%%%%%%%%%%%%%%%%%%%%%%%%%%%%%%%%%%%%%%%%%%%%%%%%%%%%%%%%%%
\section{Irreducible representations of $\Wl$ with $p$-characters of height no more than one}\label{section name-2}
In this section,
we study the simple $U_\chi(\Wl)$-modules with $r(\chi)\leq 1$.
To ease notation we put
$\Wl^{[i]}=\Wl \otimes t^i\subseteq \Wl$
and note that $\Wl=\oplus_{i=0}^{\ell}\Wl^{[i]}$ is a Lie algebra grading.
We also use notation $\Wl^{(i)}:=\sum_{j\geq i}\Wl^{[j]}$.
We record the following useful result.

\begin{Lemma}\cite[Proposition 4.1]{CT23}\label{Lemma: reduction by truncation degree}
Let $\chi\in\Wl^*$ with $\chi(W_{\ell}^{(k+1)})$ for some $0\leq k\leq \ell$,
and let $\psi:=\chi|_{W_k}\in W_k^*$,
where we identify $W_k$ with $\oplus_{i=0}^{k}\Wl^{[i]}\subseteq \Wl$ as vector spaces.
Then the full subcategory of $U_\chi(\Wl)$-{\rm mod} whose objects are the modules annihilated by $\Wl^{(k+1)}$ is equivalent to $U_\psi(W_k)$-{\rm mod}.
This full subcategory contains all simple $U_\chi(\Wl)$-modules.
\end{Lemma}
%%%%%%%%%%%%%%%%%%%%%%%%%%%%%%%%%%%%%%%%%%%%%%%%%%%%%%%%%%%%%%
\subsection{Height $-1$}
Recall that we have $\chi=0$ in this case.
Note that when $\ell=0$,
$W_\ell=W$ is the Witt algebra in our notation.
By Lemma \ref{Lemma: reduction by truncation degree},
it suffices to consider simple $U_0(W)$-modules.
Next, using Theorem \ref{Thm:Zhangherui results}(1),
one can obtain the classification of the simple $U_0(\Wl)$-modules.
%%%%%%%%%%%%%%%%%%%%%%%%%%%%%%%%%%%%%%%%%%%%%%%%%%%%%%%%%%%%%%
\subsection{Height $0$}
Given a $p$-character $\chi\in\Wl^*$ of height $0$,
we have $\chi(W_{\ell,(0)})=(0)$ and $\chi(W_{\ell,[-1]})\neq (0)$.
To classify simple $U_\chi(\Wl)$-modules,
we may assume that $\chi(e_{-1}t^{\ell})\neq 0$.
Since otherwise $\chi(\Wl^{[\ell]})=(0)$,
Lemma \ref{Lemma: reduction by truncation degree} implies that 
the problem can be reduced to classifying simple
$U_\chi(W_{\ell-1})$-modules.

\begin{Proposition}\label{prop: rchi=0,zchilambda simple}
Let $\chi\in\Wl^*$ with $r(\chi)=0$.
Then the $\chi$-reduced Verma module $\cZ_\chi(\lambda)$ is simple for all $\lambda\in\Lambda(\chi)$.
\end{Proposition}
\begin{proof}
For each $0\leq i\leq \ell-1$,
we put $b_i:=-\chi(e_{-1}t^i)/\chi(e_{-1}t^{\ell})$ as well as $x_i:=e_{-1}t^i+b_ie_{-1}t^{\ell}$.
Then the set $\{x_i;~0\leq i\leq \ell-1\}$ is a basis of $\Ker(\chi|_{W_{\ell,[-1]}})$.
It follows that
\[
W_{\ell,[-1]}=\Ker(\chi|_{W_{\ell,[-1]}})\oplus\kk e_{-1}.
\]
Then PBW theorem implies that $\cZ_\chi(\lambda)$ has basis
\[
\{x_0^{i_0}x_1^{i_1}\cdots x_{\ell-1}^{i_{\ell-1}}(e_{-1}t^{\ell})^{i_{\ell}}\otimes v_\lambda;~0\leq i_j<p~{\rm for~all}~j\}.
\]
Let $N$ be an arbitrary nonzero submodule of $\cZ_\chi(\lambda)$.
It suffices to show that $N=\cZ_\chi(\lambda)$.
Choose a nonzero element $v\in N$.
Since $x_i^{[p]}=0$ and $\chi(x_i)=0$ for all $0\leq i\leq \ell-1$,
we may thus assume that 
\[
v=\sum\limits_{j=0}^{p-1}a_jx_0^{p-1}x_1^{p-1}\cdots x_{\ell-1}^{p-1}(e_{-1}t^{\ell})^j\otimes v_\lambda,
\]
where $a_j\in\kk$ and not all zero.
For short, we will denote by $x^{\tau}:=x_0^{p-1}x_1^{p-1}\cdots x_{\ell-1}^{p-1}$.
For each $0\leq j\leq p-1$,
we have 
\[
e_0.x^{\tau}(e_{-1}t^{\ell})^j\otimes v_\lambda=(\lambda-\ell(p-1)-j)x^\tau(e_{-1}t^{\ell})^j\otimes v_\lambda.
\]
Therefore we obtain that
\[
\left( \begin{matrix}
	1&		1&		\cdots&		1\\
	\gamma&		\gamma -1&		\cdots&		\gamma -\left( p-1 \right)\\
	\vdots&		\vdots&		\ddots&		\vdots\\
	\gamma ^{p-1}&		\left( \gamma -1 \right) ^{p-1}&		\cdots&		\left( \gamma -\left( p-1 \right) \right) ^{p-1}\\
\end{matrix} \right) \left( \begin{array}{c}
	a_0x^\tau\otimes v_{\lambda}\\
	a_1x^\tau\left( e_{-1}t^{\ell} \right) \otimes v_{\lambda}\\
	\vdots\\
	a_{p-1}x^\tau\left( e_{-1}t^{\ell}\right)^{p-1}\otimes v_{\lambda}\\
\end{array} \right) 
=\left( \begin{array}{c}
	v\\
	e_0.v\\
	\vdots\\
	e_{0}^{p-1}.v\\
\end{array} \right) ,   
\]
where $\gamma=\lambda-\ell(p-1)$.
This shows that $x^\tau(e_{-1}t^{\ell})^{j_0}\otimes v_\lambda\in N$ for some $0\leq j_0\leq p-1$.
Consider the action of $e_{-1}t^\ell$ on $x^\tau(e_{-1}t^{\ell})^{j_0}\otimes v_\lambda$.
Since $\chi(e_{-1}t^\ell)\neq 0$,
we have $x^\tau\otimes v_\lambda\in N$,
so that 
\[
e_0t^\ell.x^\tau\otimes v_\lambda-\chi(e_0t^\ell)x^\tau\otimes v_\lambda\in N.    
\]
Moreover, we have
\begin{align*}
\left( e_{-1}t^{\ell} \right)^{p-1}. \left(e_0t^\ell.x^\tau\otimes v_\lambda-\chi(e_0t^\ell)x^\tau\otimes v_\lambda\right)=&\left( e_{-1}t^{\ell} \right)^{p-1}.\left[e_0t^{\ell},x^{\tau} \right] \otimes v_{\lambda}
\\
=\left( e_{-1}t^{\ell} \right) ^{p-1}\left[ e_0t^{\ell},x_{0}^{p-1} \right] x_{1}^{p-1}\cdots &x_{\ell-1}^{p-1}\otimes v_{\lambda}
=-(p-1)(e_{-1} t^{\ell})^p x_0^{p-2}x_{1}^{p-1}\cdots x_{\ell-1}^{p-1}\otimes v_{\lambda}\\
=-\left( p-1 \right) \chi & \left(e_{-1}t^{\ell}\right)^p x_{0}^{p-2}x_{1}^{p-1}\cdots x_{\ell-1}^{p-1}\otimes v_{\lambda}\ne 0.
\end{align*}
This implies $x_{0}^{p-2}x_{1}^{p-1}\cdots x_{\ell-1}^{p-1}
\otimes v_{\lambda}\in N$.
Repeating the above process, 
we conclude that $x_{1}^{p-1}\cdots x_{\ell-1}^{p-1}
\otimes v_{\lambda}\in N$.

We then proceed with similar steps, replacing the element $e_0t^{\ell}$ in the above process successively with $e_0t^{\ell-1}, e_0t^{\ell-2}, \dots, e_0t$, and finally obtain $1 \otimes v_\lambda \in N$, which implies $N = Z_\chi(\lambda)$.
\end{proof}

We have the following classification theorem on isomorphism classes of irreducible $U_\chi(\Wl)$-modules.
We will assume $\ell>0$,
because $W_0=W$ and the isomorphism classes of irreducible $U_\chi(W)$-modules were determined by Zhang (Theorem \ref{Thm:Zhangherui results}(2)).

\begin{Theorem}\label{Theorem: height 0 iso classes}
Let $\chi\in\Wl^*$ with $r(\chi)=0$.
Then we have $\cZ_\chi(\lambda)\cong\cZ_\chi(\mu)$ for any $\lambda,\mu\in\Lambda(\chi)$.
In particular,
$U_\chi(\Wl)$ has only one simple module up to isomorphism.
\end{Theorem}
\begin{proof}
For $\lambda,\mu\in\Lambda(\chi)$,
it is clear that $\lambda-\mu\in\FF_p$.
We denote by $[\lambda-\mu]\in\NN$ the minimal element such that $[\lambda-\mu]\mod p=\lambda-\mu$.
Assume that $\lambda\neq \mu$.
Consider the nonzero element $w:=(e_{-1}t^{\ell})^{[\lambda-\mu]}\otimes v_\lambda\in\cZ_\chi(\lambda)$.
For $1\leq i\leq p-2$,
we have
\begin{align}\label{formula of proof-1}
e_i.w&=e_i.(e_{-1}t^{\ell})^{[\lambda-\mu]}\otimes v_\lambda\\\nonumber
=&\left( e_{-1}t^{\ell} \right) ^{\left[ \lambda -\mu \right]}e_i\otimes v_{\lambda}
-\left( i+1 \right) \left( \lambda -\mu \right) \left( e_{-1}t^{\ell}\right) ^{\left[ \lambda -\mu \right] -1}
\left( e_{i-1}t^{\ell} \right) \otimes v_{\lambda}.   
\end{align}
Clearly, $e_i\otimes v_\lambda=0$.
Also, by assumption on $\ell\geq 1$, $e_{i-1}t^{\ell}$ is $p$-nilpotent. This yields $e_{i-1}t^{\ell}\otimes v_{\lambda}=0$.
As a result, $e_i.w=0$.
Moreover,
for $1\leq i\leq p-2$ and $1\leq j\leq \ell$,
we have
\begin{align}\label{formula of proof-2}
e_0.w =\mu w\,,\,\left( e_0t^j \right).\omega =\chi_j \omega\,,\,\left( e_it^j \right).w =0.    
\end{align}
Now the universal property of tensor product gives rise to a nonzero homomorphism
\[
\cZ_\chi(\mu)\rightarrow \cZ_\chi(\lambda);~v_\mu\mapsto w.
\]
Then Proposition \ref{prop: rchi=0,zchilambda simple} implies that the homomorphism has to be an isomorphism.
The last assertion immediately follows from Propostion \ref{prop: simpe module is image of zchilambda}.
\end{proof}

%%%%%%%%%%%%%%%%%%%%%%%%%%%%%%%%%%%%%%%%%%%%%%%%%%%%%%%%%%%%%%
\subsection{Height $1$}
In this section, we will consider the case that $r(\chi)=1$.
As before, we always assume that $\ell\geq 1$ (see Theorem \ref{Thm:Zhangherui results}(3)).
Note that $\chi(W_{\ell,(1)})=(0)$.
In fact by Lemma \ref{Lemma: reduction by truncation degree},
we can assume further that $\chi(e_{-1}t^{\ell})$ and $\chi(e_0t^{\ell})$ are not simultaneously zero.

\begin{Proposition}\label{prop: rchi=1,simple}
Let $\chi\in\Wl^*$ with $r(\chi)=1$.
Then the $\chi$-reduced Verma module $\cZ_\chi(\lambda)$ is simple for all $\lambda\in\Lambda(\chi)$.
\end{Proposition}
\begin{proof}
If $\chi(e_{-1}t^{\ell})\neq 0$,
then assertion follows from a similar proof of Proposition \ref{prop: rchi=0,zchilambda simple}.

Suppose that $\chi(e_{-1}t^{\ell})=0$.
We know that $\chi(e_{0}t^{\ell})\neq 0$ by assumption.
Let $N$ be an arbitrary nonzero submodule of $\cZ_\chi(\lambda)$.
We need to show that $N=\cZ_\chi(\lambda)$.
If $\chi(W_{\ell,[-1]})\neq (0)$,
then there exists a unique integer $0\leq m\leq \ell-1$ such that $\chi(e_{-1}t^m)\neq 0$ and $\chi(e_{-1}t^i)=0$ for all $m<i\leq \ell$.
We divide the proof into three cases.

Case $1$: $\chi(W_{\ell,[-1]})=(0)$.
Take a nonzero element $u\in N$.
By using the action of the elements $e_{-1}t^i$ we can get 
\[
 e_{-1}^{p-1}(e_{-1}t)^{p-1}\cdots (e_{-1}t^{\ell})^{p-1}\otimes v_\lambda\in N.
\]
Then we have
\begin{align*}
e_1t^{\ell}.e_{-1}^{p-1}(e_{-1}t)^{p-1}\cdots(e_{-1}t^{\ell})^{p-1}\otimes & v_{\lambda}
\\
=e_{-1}^{p-1}(e_1t^{\ell})(e_{-1}t)^{p-1}\cdots(e_{-1}&t^{\ell})^{p-1}\otimes  v_{\lambda}
\\
-2(p-1) e_{-1}^{p-2}(e_0t^{\ell})&(e_{-1}t)^{p-1}\cdots( e_{-1}t^{\ell})^{p-1}\otimes v_{\lambda}
\\
+2\binom{p-1}{2} e_{-1}^{p-3}&(e_{-1}t^{\ell})(e_{-1}t)^{p-1}\cdots(e_{-1}t^{\ell})^{p-1}\otimes v_{\lambda}
\\
= -2(p-&1)\chi(e_0t^{\ell}) 
e_{-1}^{p-2}(e_{-1}t)^{p-1}\cdots(e_{-1}t^{\ell})^{p-1}\otimes v_{\lambda}.
\end{align*}
This shows that $e_{-1}^{p-2}(e_{-1}t)^{p-1}\cdots(e_{-1}t^{\ell})^{p-1}\otimes v_{\lambda}\in N$.
Repeating the above process, 
we conclude that 
\[
(e_{-1}t)^{p-1}\cdots(e_{-1}t^{\ell})^{p-1}\otimes v_{\lambda}\in N. 
\]
We next proceed with similar computation,
replacing the element $e_1t^{\ell}$ in the above process successively with $e_1t^{\ell-1}, e_1t^{\ell-2}, \dots, e_1$,
and finally obtain $1 \otimes v_\lambda \in N$, which implies $N = Z_\chi(\lambda)$.

Case $2$: $m=0$. In this case, we have $\chi(e_{-1})\neq 0$ and $\chi(e_{-1}t^i)=0$ for all $1\leq i\leq \ell$.
Using arguments similar to the proof of Proposition \ref{prop: rchi=0,zchilambda simple}, one can show
\[
(e_{-1}t)^{p-1}(e_{-1}t^2)^{p-1}\cdots(e_{-1}t^{\ell})^{p-1}\otimes v_\lambda\in N.
\]
Then by Case $1$, we also have $1\otimes v_\lambda\in N$,
so that $N=\cZ_\chi(\lambda)$.

Case $3$: $m>0$.
For $0\leq i\leq m-1$,
we set
\[
x_i:=e_{-1}t^i-\chi(e_{-1}t^i)/\chi(e_{-1}t^m)e_{-1}t^m,
\]
so that $\chi(x_i)=0$.
Moreover, the set $\{x_i,e_{-1}t^j;~0\leq i<m, m\leq j\leq \ell\}$ 
is a basis of $W_{\ell,[-1]}$ and $\cZ_\chi(\lambda)$ has basis
\[
\{x_0^{i_0}x_1^{i_1}\cdots x_{m-1}^{i_{m-1}}(e_{-1}t^{m})^{i_m}\cdots (e_{-1}t^{\ell})^{i_\ell}\otimes v_\lambda;~0\leq i_j<p~{\rm for~all}~j\}.
\]
Using a similar argument as in Proposition \ref{prop: rchi=0,zchilambda simple}, we obtain 
\[
x_0^{p-1}x_1^{p-1}\cdots x_{m-1}^{p-1}(e_{-1}t^{m+1})^{p-1}\cdots (e_{-1}t^{\ell})^{p-1}\otimes v_\lambda\in N.
\]
Then we have
\begin{align*}
e_1t^{\ell}.x_0^{p-1}x_1^{p-1}\cdots x_{m-1}^{p-1}(e_{-1}t^{m+1})^{p-1}\cdots (e_{-1}t^{\ell})^{p-1}\otimes  v_{\lambda}&\\
=x_0^{p-1}(e_1t^{\ell})x_1^{p-1}\cdots x_{m-1}^{p-1}(e_{-1}t^{m+1})^{p-1}\cdots  (e_{-1}&t^{\ell})^{p-1}\otimes v_{\lambda}\\
-2(p-1)x_0^{p-2}(e_0t^{\ell})x_1^{p-1}\cdots x_{m-1}^{p-1}(e_{-1}&t^{m+1})^{p-1}\cdots (e_{-1}t^{\ell})^{p-1}\otimes v_{\lambda}
\\
+2\binom{p-1}{2}x_0^{p-3}(e_{-1}t^{\ell})x_1^{p-1}\cdots&x_{m-1}^{p-1}(e_{-1}t^{m+1})^{p-1}\cdots (e_{-1}t^{\ell})^{p-1}\otimes v_{\lambda}\\
=-2(p-1)\chi(e_0t^{\ell})&x_0^{p-2}x_1^{p-1}\cdots x_{m-1}^{p-1}(e_{-1}t^{m+1})^{p-1}\cdots (e_{-1}t^{\ell})^{p-1}\otimes v_{\lambda}.
\end{align*}
This yields
\[
x_0^{p-2}x_1^{p-1}\cdots x_{m-1}^{p-1}(e_{-1}t^{m+1})^{p-1}\cdots (e_{-1}t^{\ell})^{p-1}\otimes v_{\lambda}\in N.
\]
Repeating the above process, 
we obtain that
\[
x_1^{p-1}\cdots x_{m-1}^{p-1}(e_{-1}t^{m+1})^{p-1}\cdots (e_{-1}t^{\ell})^{p-1}\otimes v_{\lambda}\in N.    
\]
As before,
we can proceed with a similar computation,
replacing in the above process the element $e_1t^{\ell}$ successively with $e_1t^{\ell-1}, e_1t^{\ell-2}, \dots, e_1t^{\ell-m+1}$,
and we obtain
\[
(e_{-1}t^{m+1})^{p-1}\cdots (e_{-1}t^{\ell})^{p-1}\otimes v_{\lambda}\in N.  
\]
Moreover,
a straightforward computation gives
\[
e_1t^{\ell-m-1}.(e_{-1}t^{m+1})^{p-1}\cdots (e_{-1}t^{\ell})^{p-1}\otimes  v_{\lambda}=-2(p-1)\chi(e_0t^{\ell})(e_{-1}t^{m+1})^{p-2}\cdots (e_{-1}t^{\ell})^{p-1}\otimes v_{\lambda}\in N.
\]
Similarly, we can get $1\otimes v_\lambda\in N$.
As a result, $N=\cZ_\chi(\lambda)$.
\end{proof}

We have the following classification theorems on isomorphism classes of irreducible $U_\chi(\Wl)$-modules. 
\begin{Theorem}\label{thm:height 1 iso 1}
Let $\chi\in\Wl^*$ with $r(\chi)=1$ and $\chi(e_{-1}t^{\ell})\neq 0=\chi(e_{0}t^{\ell})$. 
Then we have $\cZ_\chi(\lambda)\cong\cZ_\chi(\mu)$ for any $\lambda,\mu\in\Lambda(\chi)$.
In particular,
$U_\chi(\Wl)$ has only one simple module up to isomorphism.
\end{Theorem}
\begin{proof}
For $\lambda,\mu\in\Lambda(\chi)$, we have by \eqref{Lambda chi} $\lambda-\mu\in\FF_p$.
Note that $\chi(e_0t^{\ell})=0$.
It follows that \eqref{formula of proof-1} and \eqref{formula of proof-2}
still hold true.
Then we can use the same argument as in the proof of Theorem \ref{Theorem: height 0 iso classes}.
\end{proof}

\begin{Theorem}\label{thm:height 1 iso 2}
Let $\chi\in\Wl^*$ with $r(\chi)=1$ and $\chi(e_{0}t^{\ell})\neq 0$.
Then the set $\{\cZ_\chi(\lambda);~\lambda\in\Lambda(\chi)\}$ exhausts all non-isomorphic simple $U_\chi(\Wl)$-modules.
\end{Theorem}
\begin{proof}
According to Proposition \ref{prop: simpe module is image of zchilambda} and Proposition \ref{prop: rchi=1,simple},
we need to show that $\cZ_\chi(\lambda)\cong\cZ_\chi(\mu)$ if and only if $\lambda=\mu$.
Let $\varphi: \cZ_\chi(\lambda)\rightarrow\cZ_\chi(\mu)$ be a $\Wl$-module isomorphism.
Set $I_p:=\{0,1,\dots,p-1\}\subseteq\ZZ$.
We can write
\[
w:=\varphi(1\otimes v_\lambda)=\sum\limits_{k=0}^{p-1}\sum\limits_{I=(i_1,\dots,i_\ell)\in I_p^{\ell}}a_{k,I}e_{-1}^k(e_{-1}t)^{i_1}\cdots (e_{-1}t^{\ell})^{i_\ell}\otimes v_\mu.
\]
Since $1\otimes v_\lambda$ is a vector annihilated by all $e_1t^i$,
so is $w$.
For the following computations observe that
\[
(e_1t^{\ell})e_{-1}^k=e_{-1}^k(e_1t^{\ell})-2ke_{-1}^{k-1}(e_0t^{\ell})+k(k-1)e_{-1}^{k-2}(e_{-1}t^{\ell})
\]
for $0\leq k\leq p-1$.
Then 
\begin{align}\label{jisuanzhong-1}\nonumber
0=e_1t^{\ell}.w= -2\sum\limits_{k=0}^{p-1}&\sum\limits_{I=( i_1,\dots ,i_\ell) in I_p^{\ell}}a_{k,I}\chi( e_0t^{\ell}) ke_{-1}^{k-1}(e_{-1}t)^{i_1}\cdots (e_{-1}t^{\ell})^{i_\ell}\otimes v_{\mu}\\
&+\sum_{k=0}^{p-1}\sum\limits_{I=i_1,\dots ,i_\ell) \in I_p^{\ell}}a_{k,I}k(k-1) e_{-1}^{k-2}(e_{-1}t)^{i_1}\cdots (e_{-1}t^{\ell})^{i_\ell+1}\otimes v_{\mu}.
\end{align}
The terms containing $e_{-1}^{p-2}$ in \eqref{jisuanzhong-1} are
\[
-2(p-1)\chi(e_0t^{\ell})\sum\limits_{I=(i_1,\dots,i_\ell)\in I_p^{\ell}}a_{p-1,I}e_{-1}^{p-2}(e_{-1}t)^{i_1}\cdots (e_{-1}t^{\ell})^{i_\ell}\otimes v_{\mu},
\]
which is zero.
This implies that $a_{p-1,I}=0$ for all $I\in I_p^\ell$.
We continue to examine the terms involving $e_{-1}^{p-3}$ in \eqref{jisuanzhong-1} with the use of $a_{p-1,I}=0$.
This gives
\begin{align*}
-2(p-2)\chi(e_0t^{\ell})\sum\limits_{I=(i_1,\dots,i_\ell)\in I_p^\ell}a_{p-2,I} &e_{-1}^{p-3}(e_{-1}t)^{i_1}\cdots(e_{-1}t^{\ell})^{i_\ell}\otimes v_{\mu}
\\
+(p-1)(p-2)\sum\limits_{I=(i_1,\dots,i_\ell)\in I_p^\ell}&a_{p-1,I} e_{-1}^{p-3}(e_{-1}t)^{i_1}\cdots (e_{-1}t^{\ell})^{i_\ell+1}\otimes v_{\mu}
\\
=-2(p-2)\chi(e_0t^{\ell})\sum\limits_{I=(i_1,\dots,i_\ell)\in I_p^\ell}a_{p-2,I} &e_{-1}^{p-3}(e_{-1}t)^{i_1}\cdots(e_{-1}t^{\ell})^{i_\ell}\otimes v_{\mu}=0,
\end{align*}
which implies that $a_{p-2,I}=0$ for all $I\in I_p^\ell$.
Proceeding in this manner and successively examining the terms involving $e_{-1}^{p-4},\cdots, e_{-1}^0$ in \eqref{jisuanzhong-1}, we conclude that $a_{k,I}=0$ for all $k\geq 1$ and $I\in I_p^{\ell}$.
Thus, we may rewrite $w$ as follows:
\[w=\sum_{k=0}^{p-1}\sum_{I=(i_2,\dots,i_\ell)\in I_p^{\ell-1}}
b_{k,I}(e_{-1}t)^k(e_{-1}t^2)^{i_2}\cdots (e_{-1}t^{\ell})^{i_\ell}\otimes v_{\mu},
\]
where $b_{k,I}\in\kk$ are not all zero.
Using the fact that $e_1t^{\ell-1}.w=0$ and the following identity
\[
e_1t^{\ell-1}(e_{-1}t)^k=(e_{-1}t)^k e_1t^{\ell-1}-2k(e_{-1}t)^{k-1}e_0t^\ell\,\,,\,\,0\leq k\leq p-1, 
\]
we obtain
\[
0=e_1t^{\ell-1}.w=-2\chi(e_0t^\ell)\sum_{k=0}^{p-1}\sum_{I=(i_2,\dots,i_\ell)\in I_p^{\ell-1}}
b_{k,I}k(e_{-1}t)^{k-1}(e_{-1}t^2)^{i_2}\cdots (e_{-1}t^\ell)^{i_\ell}\otimes v_{\mu}.
\]
This implies that $b_{k,I}=0$ for all $k\geq 1$ and $I\in I_{p}^{\ell-1}$.
Proceeding in this way, we finally get $w=c(1\otimes v_\mu)$ with $c\ne 0$.
Then we have
\[
 \mu w=e_0.w=e_0. \varphi(1\otimes v_\lambda)=\lambda w,  
\]
forcing that $\lambda=\mu$.
We complete the proof.
\end{proof}
%%%%%%%%%%%%%%%%%%%%%%%%%%%%%%%%%%%%%%%%%%%%%%%%%%%%%%%%%%%%%%
\section{Irreducible representations of $\Wl$ with $p$-characters of height greater than one}\label{section name-3}
In this section,
we study a class of irreducible representations of $U_\chi(\Wl)$ with $r(\chi)>1$.
More precisely,
we assume that $\chi(W_{\ell,(r)})=(0)$ and $\chi(e_{r-1}t^\ell)\neq 0$.
Clearly, $r(\chi)=r$ and we define
$s:=\lfloor r/2\rfloor$ the largest integer no more than $r/2$.
\begin{Lemma}\label{strade's condition}
Let $\chi\in\Wl^*$ with $1<r<p-1$ and $\chi(e_{r-1}t^\ell)\neq 0$.
For each $0\leq k\leq s$, 
there exists $y_{k,0},y_{k,1},\dots,y_{k,\ell}\in W_{\ell,(r-k)}$ such that the
following holds.
\begin{enumerate}
    \item[(1)] $\chi([e_{k-1}t^i,y_{k,j}])\neq 0$ if $i=j$, and $\chi([e_{k-1}t^i,y_{k,j}])=0$ if $i\neq j$,
    \item[(2)] $[e_{k-1}t^i,y_{k,j}]$, $[[e_{k-1}t^i,y_{k,j}],e_{k-1}t^l]\in W_{\ell,(k)}$ for all $i,j,l$,
    \item[(3)] $\chi([W_{\ell,(k)},W_{\ell,(r-k)}]+[\mathfrak{r},\mathfrak{r}])=(0)$, where $\mathfrak{r}$ is the Lie subalgebra generated by\\
    $$W_{\ell,(r-k)}+\sum_{i,j=0}^\ell \kk[e_{k-1}t^i,y_{k,j}].$$
\end{enumerate}
\end{Lemma}
\begin{proof}
Fix $0\leq k\leq s$.
For $0\leq i\leq \ell-1$,
set $b_i:=-\chi(e_{r-1}t^i)/\chi(e_{r-1}t^\ell)$.
Now, we define inductively the elements $y_{k,j}\in W_{\ell,(r-k)}$ for $j=0,1,\dots,\ell$ via
\[
 y_{k,j}:=e_{r-k}t^{\ell-j}+\sum\limits_{i=0}^{j-1}b_{\ell-j+i}y_{k,i}.
\]
Observe that $r\geq 2$ and $r-1\geq k$.
Then
\[
[e_{k-1}t^i,y_{k,j}]\in W_{\ell,(r-1)}\subseteq W_{\ell,(k)}
\]
and
\[ 
[[e_{k-1}t^i,y_{k,j}],e_{k-1}t^l]\in W_{\ell,(r-2+k)} \subseteq W_{\ell,(k)}, 
\]
so that the condition (2) is satisfied.

Moreover,
we have
\[
[W_{\ell,(k)},W_{\ell,(r-k)}]+[\mathfrak{r},\mathfrak{r}]\subseteq W_{\ell,(r)}+[W_{\ell,(r-k)}+W_{\ell,(r-1)},W_{\ell,(r-k)}+W_{\ell,(r-1)}]\subseteq W_{\ell,(r)}    
\]
It follows that $\chi([W_{\ell,(k)},W_{\ell,(r-k)}]+[\mathfrak{r},\mathfrak{r}])=(0)$.
This gives condition (3).

For condition (1), it is easy to check that 
\begin{align}\label{comp}
\chi([e_{k-1}t^i,y_{k,i}])=\chi([e_{k-1}t^i,e_{r-k}t^{\ell-i}])=(r+1-2k)\chi(e_{r-1}t^\ell)\neq 0       
\end{align}
for all $0\leq i\leq \ell$ and $[e_{k-1}t^i,y_{k,j}]=0$ for all $0\leq j<i\leq \ell$.
We shall prove by induction on $j$ that $\chi([e_{k-1}t^i,y_{k,j}])=0$ for all $0\leq i<j\leq \ell$. 
One computes by using \eqref{comp} that
\begin{align*}
\chi([e_{k-1}t^i,y_{k,i+1}])=&\chi([e_{k-1}t^i,e_{r-k}t^{\ell-i-1}+\sum\limits_{l=0}^{i}b_{\ell-i-1+l}y_{k,l}])\\
=&\chi([e_{k-1}t^i,e_{r-k}t^{\ell-i-1}+b_{\ell-1}y_{k,i}])\\
=&\chi([e_{k-1}t^i,e_{r-k}t^{\ell-i-1}])+b_{\ell-1}(r+1-2k)\chi(e_{r-1}t^{\ell})\\
=&(r+1-2k)\left(\chi(e_{r-1}t^{\ell-1})+b_{\ell-1}\chi(e_{r-1}t^{\ell})\right)=0.
\end{align*}
Suppose that $\chi([e_{k-1}t^i,y_{k,j}])=0$ for all $i+1\leq j\leq m$.
Then
\begin{align*}
\chi([e_{k-1}t^i,y_{k,m+1}])=&\chi([e_{k-1}t^i,e_{r-k}t^{\ell-m-1}+\sum\limits_{l=0}^{m}b_{\ell-m-1+l}y_{k,l}])\\
=&\chi([e_{k-1}t^i,e_{r-k}t^{\ell-m-1}+b_{\ell-m-1+i}y_{k,i}])\\
=&\chi([e_{k-1}t^i,e_{r-k}t^{\ell-m-1}])+b_{\ell-m-1+i}(r+1-2k)\chi(e_{r-1}t^{\ell})\\
=&(r+1-2k)\left(\chi(e_{r-1}t^{\ell-m-1+i})+b_{\ell-m-1+i}\chi(e_{r-1}t^{\ell})\right)=0.
\end{align*}
We have now verified that the condition (1) holds.
\end{proof}

\begin{Theorem}\label{thm:height greater than one}
Let $\chi\in\Wl^*$ with $1<r=r(\chi)<p-1$ and $\chi(e_{r-1}t^{\ell})\neq 0$.
Then up to isomorphism there is a unique simple $U_\chi(\Wl)$-module $S$ represented by 
\begin{align}\label{higher height simple module}
 U_\chi(W_\ell)\otimes_{U_{\chi}(W_{\ell,(s)})}\kk_\chi,   
\end{align}
where $s=\lfloor r/2\rfloor$ and $\kk_\chi$ is the one-dimensional simple $U_\chi(W_{\ell,(s)})$-module.
In particular, $\dim_\kk S=p^{(\ell+1)(s+1)}$.
\end{Theorem}
\begin{proof}
Since $\chi(W_{\ell,(s)})=(0)$,
there exists a unique simple $U_{\chi}(W_{\ell,(s)})$-module $\kk_\chi$ which is completely determined by $\chi$.
For each $0\leq k\leq s$, we see that $W_{\ell,(r-k)}$ is an ideal of $W_{\ell,(k)}$ and $W_{\ell,(k-1)}$ and $W_{\ell,(k-1)}=W_{\ell,(k)}+\sum_{i=0}^{\ell}\kk e_{k-1}t^i$.
Let $M_{k-1}$ be a simple $U_\chi(W_{\ell,(k-1)})$ and $M_k\subseteq M_{k-1}$ be an irreducible $W_{\ell,(k)}$-submodule.
Thanks to Lemma \ref{strade's condition},
we may apply the Main Theorem of \cite{Strade77} to see that there is $W_{\ell,(k-1)}$-module isomorphism
\begin{align}\label{Mk-1 iso Mk}
M_{k-1}\cong U_\chi(W_{\ell,(k-1)})\otimes_{U_\chi(W_{\ell,(k)})}M_k.   
\end{align}
To obtain the simple $U_{\chi}(W_{\ell})$-module \eqref{higher height simple module},
we start with the unique simple $U_{\chi}(W_{\ell,(s)})$-module $\kk_\chi$.
The observation above provides a sequence
\[
\kk_\chi=M_s\subseteq M_{s-1}\subseteq\cdots \subseteq M_0\subseteq M_{-1}
\]
such that \eqref{Mk-1 iso Mk} holds for all $0\leq k\leq s$.
In particular,  $M_{-1}\cong U_\chi(W_\ell)\otimes_{U_{\chi}(W_{\ell,(s)})}\kk_\chi$
is the unique simple $U_\chi(\Wl)$-module.
\end{proof}

\noindent
\textbf{Acknowledgment.}
This work is supported by
the Natural Science Foundation of Hubei Province (No. 2025AFB716).
%%%%%%%%%%%%%%%%%%%%%%%%%%%%%%%%%%%%%%%%%%%%%%%%%%%%%%%%%%%%%%
%%%%%%%%%%%%%%%%%%%%%%%%%%%%%%%%%%%%%%%%%%%%%%%%%%%%%%%%%%%%%%
%%%%%%%%%%%%%%%%%%%%%%%%%%%%%%%%%%%%%%%%%%%%%%%%%%%%%%%%%%%%%%

\end{document}